\documentclass{amsart}
\usepackage{graphicx}
\usepackage{amssymb}
\usepackage{array}
\usepackage{enumerate}
\usepackage{url}
\usepackage{algorithm}
\usepackage{dsfont}

\DeclareMathOperator{\re}{Re}

\newtheorem{theorem}{Theorem}

\newtheorem{lemma}[theorem]{Lemma}

\numberwithin{theorem}{section}
\theoremstyle{remark}
\newtheorem*{remark}{Remark}

\makeatletter
\let\@@pmod\pmod
\DeclareRobustCommand{\pmod}{\@ifstar\@pmods\@@pmod}
\def\@pmods#1{\mkern 2mu({\operator@font mod}\mkern 3mu#1)}
\makeatother

\begin{document}

\title[Computing quadratic Dirichlet $L$-functions]{A fast
amortized algorithm for computing quadratic Dirichlet $L$-functions}
\author[G.A. Hiary]{Ghaith A. Hiary}
\thanks{Preparation of this material is partially supported by
the National Science Foundation under agreements No. 
 DMS-1406190. The author is pleased to thank
  ICERM where part of this work was conducted.}
\address{Department of Mathematics, The Ohio State University, 231 West 18th
Ave, Columbus, OH 43210.}
\email{hiary.1@osu.edu}

\begin{abstract}
An algorithm to compute Dirichlet $L$-functions 
for many quadratic characters is derived. The algorithm
is optimal (up to logarithmic factors) provided that 
the conductors of 
the characters under consideration span a dyadic window.
\end{abstract}

\maketitle

\section{Introduction} \label{intro}

Let $s=\sigma+it$ where $\sigma$ and $t$ are real, 
$q>1$ a positive integer, $\chi$ 
a primitive quadratic Dirichlet character modulo $q$, 
and $L(s,\chi)$ the corresponding Dirichlet $L$-function.
In this article, we derive an algorithm 
to compute $L(s,\chi)$ for all quadratic characters $\chi$ with conductor 
$q \in [Q,Q+\Delta)$.
The algorithm is essentially optimal provided that
$\Delta$ is of comparable size to $Q$, and so can be viewed as an
``Odlyzko--Sch\"onhage algorithm''in the $q$-aspect (rather than 
the $t$-aspect). This represents the first progress of its kind
in amortized algorithms for $L$-functions 
since the Odlyzko--Sch\"onhage
algorithm for the zeta function. 
To state the main theorem, let $d(q)$ denote the number of 
divisors of $q$. 

\begin{theorem}\label{main thm}
Given any $t$, there are
constants $A_j:=A_j(t)$, $j\in\{1,2,3,4\}$, such that for every
$\epsilon>0$ and $1\le \Delta < Q/2$,
the function $L(1/2+it,\chi)$ can be numerically evaluated
to within $\epsilon$ for any quadratic characters $\chi$ with 
conductor $q\in [Q,Q+\Delta)$ using 
$\le A_1 d(q) \log(Q/\epsilon)$ 
elementary arithmetic operations 
on numbers of $\le A_2 \log(Q/\epsilon)$ bits
provided that a precomputation requiring $\le A_3 Q\log^4(Q/\epsilon)$ 
operations
and $\le A_4 Q\log^4(Q/\epsilon)$ bits of storage is
first performed. 
\end{theorem}

The average of the divisor function $d(q)$ 
over $q\in [Q,Q+\Delta)$, $\Delta\asymp Q$, is $\ll \log
Q$ (see e.g.\ \cite[\textsection{12.1}]{titchmarsh}).
Hence, Theorem~\ref{main thm} 
says that the average complexity of computing 
$L(1/2+it,\chi)$ is poly-log in $Q$ and
$1/\epsilon$. 
This is a large improvement over the
approximate functional equation which, in comparison,
requires $Q^{1/2+o_{\epsilon}(1)}$ time
per evaluation.\footnote{There is currently 
no algorithm for computing $L(s,\chi)$ at a single $\chi$ that
improves on the approximate functional equation in the $q$-aspect
except for an algorithm due to
the author in the case $\chi$ is to a powerfull modulus, but this case 
excludes quadratic characters.} 

The assumption in the theorem that 
$\Delta < Q/2$ is not essential, and
$\Delta \ll Q$ still works simply by dividing into further dyadic sections if
needed. Moreover, the upper bound 
$Q/2$ can be replaced by $c\, Q$ for any positive constant
$c<1$, though the dependence of the 
constants $A_j$ on $c$ would then be of the form
$1/\log(1/c)$.

The theorem is stated for $s$ on the critical line 
because our main interest is in computing low zeros of quadratic
Dirichlet $L$-functions. 
But the theorem applies just as well
to other values of $s$; e.g.\ to $s=1$.

Given the wide availability of fast
computer memory nowadays, the algorithm underlying Theorem~\ref{main thm} 
is practical. In fact, using a clever implementation
one can hope to save at least one power of the logarithm (and probably two
powers in the case of storage space)
on powers of the logarithm that appear in the theorem, and to
remove the dependence on the requested accuracy 
$\epsilon$ in some cases.

Of course, Theorem~\ref{main thm} is not useful if one is interested in evaluating 
$L(1/2+it,\chi)$ for fewer than $\sqrt{Q}$
different $\chi$, since in this case the approximate functional equation fares
better.
However, in some applications, such as the locating of low zeros
of all quadratic Dirichlet $L$-functions with conductor in a dyadic section, 
our theorem is suitable. And even if one is interested in smaller sections, 
the theorem may still be of utility. 
For the proof of the theorem yields that 
if $\Delta=Q^{1-\delta}$, then the average time to compute all 
$L(s,\chi)$ with conductor in $[Q,Q+\Delta)$ is 
$Q^{\delta+o_{\epsilon}(1)}$ operations. 
 This is an improvement over the approximate functional equation 
if $\delta <1/2$.

The strategy behind Theorem~\ref{main thm} is roughly this. 
First, we use the well-known formula involving quadratic Gauss sums 
to express each quadratic character $\chi$ as a Fourier series. 
This representation of $\chi$ implicitly involves a single application of quadratic
reciprocity, which is important as without it our algorithm achieves no savings.
The resulting main sum for $L(s,\chi)$ is then an exponential sum 
of length $Q^{1+o_{\epsilon}(1)}$ terms.
The advantage of having replaced $\chi$ by 
a quadratic Gauss sum is that now each $\chi$ corresponds to 
a point on the unit circle 
where we want to evaluate the exponential sum.
So we proceed similarly 
to the Odlyzko--Sch\"onhage algorithm from here, 
applying a fast multipole method to numerically evaluate 
 the exponential sum at all the required points on the unit circle 
in $(\Delta+Q)^{1+o_{\epsilon}(1)}$ time and 
requiring $(\Delta+Q)^{1+o_{\epsilon}(1)}$  bits of storage space.

Actually, our algorithm is slightly more complicated 
than this because 
there is dependence on $q$ through
a weighting function $V_z(w)$ in the main sum; see formula 
\eqref{afe Z}. But this dependence is 
relatively weak 
and is easily removed via a Taylor expansion. Another complication 
is that representing $\chi$ by
a quadratic Gauss sum when $q$ is composite and varying in a range
necessitates using an inclusion-exclusion formula; see formula \eqref{inc-exc}.

Interestingly, Theorem~\ref{main thm} might generalize to
 certain families of elliptic curve $L$-functions, though there are important
 differences with the present case.
This will be the subject of future work.

Sections 2--5 present the essential ingredients of the algorithm in
Theorem~\ref{main thm}.
Section \ref{main thm proof} contains a proof of the theorem.
Section \ref{improvements} discusses improvements to the algorithm. And Section 
\ref{proofs of lemmas} contains proofs of auxiliary lemmas. 

\vspace{3mm}
\noindent
\textbf{Acknowledgment.} This paper was motivated by a question of Brian Conrey
about computing with various families of $L$-functions
during the ICERM special semester in Fall 2015. I would like to thank Leslie
Greengard for pointing out references about the non-uniform fast Fourier
transform.

\section{The main sum}
Let $\chi$ 
be a non-principal quadratic Dirichlet character
of modulus $q > 1$. 
The Dirichlet $L$-function associated with $\chi$ is 
\begin{equation}
L(s,\chi) := \sum_{n=1}^{\infty} \frac{\chi(n)}{n^s},\qquad \sigma > 0.
\end{equation}
The completed $L$-function is 
$\Lambda(s,\chi) := q^{s/2} \gamma(s) L(s,\chi)$,
where
\begin{equation}
\gamma(z) := \pi^{-z/2} \Gamma\left(\frac{z+\mathfrak{a}}{2}\right),\qquad
\mathfrak{a}:= \frac{1-\chi(-1)}{2}\in \{0,1\}. 
\end{equation}
If $\chi$ is primitive, then
the functional equation takes on the simple form 
$\Lambda(s,\chi) = \Lambda(1-s,\chi)$. 

Using the method of the smoothed functional equation (see
e.g.\ \cite{rubinstein}),
one obtains the following formula to compute $L(s,\chi)$. 
Let $\Gamma(z,w)$ denote the incomplete gamma function, 
\begin{equation}\label{incomp gamma}
\Gamma(z,w) = w^z \int_1^{\infty} e^{-wy} y^{z-1}dy,\qquad
\re(w)>0.
\end{equation}
And let
\begin{equation}\label{V def}
V_z(w) := \frac{\Gamma(z/2,w)}{\Gamma(z/2)},\qquad
\rho(s) :=q^{1/2-s}\frac{\gamma(1-s)}{\gamma(s)}.
\end{equation}
Then we have
\begin{equation}\label{afe L}
L(s,\chi) = \sum_{n=1}^{\infty} \frac{\chi(n)}{n^s}
V_{s+\mathfrak{a}}(\pi n^2/q) + \rho(s) 
\sum_{n=1}^{\infty} \frac{\chi(n)}{n^{1-s}}
V_{1-s+\mathfrak{a}}(\pi n^2/q).
\end{equation}
The weighting function $V$ ensures the absolute convergence of the series
in \eqref{afe L}.
So the series provides an analytic continuation of $L(s,\chi)$ to the entire
complex plane. 
Henceforth, we restrict to the critical line $s=1/2+it$.

As a consequence of the functional equation, $\Lambda(1/2+it,\chi)$ is
real-valued for real $t$. This desirable feature of $\Lambda$ 
simplifies 
locating roots on the critical line by looking for sign changes of
$\Lambda(1/2+it)$.
However, unlike $L(1/2+it,\chi)$ itself, the function 
$\Lambda(1/2+it,\chi)$ experiences exponential 
decay with $t$ due to the built-in $\gamma$ factor. This 
makes $\Lambda$ unsuitable for numerical computations when 
$t$ is large; e.g.\ very
high accuracy is needed to confirm sign changes.
In view of this, one is motivated to introduce the functions
\begin{equation}\label{afe}
\theta(t,\mathfrak{a}) := \arg\left[q^{it/2} \gamma(1/2+it)\right],\qquad
Z(t,\chi) := e^{i\theta(t,\mathfrak{a})}L(1/2+it,\chi).
\end{equation}
So $Z(t,\chi)$ is of the same modulus as $L(1/2+it,\chi)$ while also 
analytic and real-valued for real $t$.
Furthermore, the formula \eqref{afe L} now reads
\begin{equation}\label{afe Z}
Z(t,\chi) = 2\re \left[ e^{i\theta(t,\mathfrak{a})}
\sum_{n=1}^{\infty} \frac{\chi(n)}{n^{1/2+it}} 
V_{1/2+it+\mathfrak{a}}(\pi n^2/q)\right]. 
\end{equation}

As pointed out in \cite{rubinstein}, 
there is still a problem of catastrophic cancellation in \eqref{afe Z} because 
the terms in the series get very large and oscillatory if $t$ is large. 
To fix the problem, one can use the same modification 
in \cite{rubinstein},
after which the algorithm that we describe here works even when $t$ is large. 
Our focus, though, will be on computing low zeros of $Z(t,\chi)$. 
So we might imagine without loss of generality that $t\in [0,1]$ say
(or $t\in [0,c_1]$ for any fixed $c_1>0$). 
Then the problem of catastrophic cancellation no longer arises, and 
the formula \eqref{afe Z} is good enough as is. 
One could also use formula \eqref{afe L} and work with $\Lambda(1/2+it,\chi)$
directly. In either case, the dependence of the constants $A_j$ on $t$ is 
roughly of the form $e^{\pi t/4}$.

There is a sharp drop-off in the size of the weighting function $V$ when $n$ is
around the square-root of the analytic conductor
$\mathfrak{q}:=\mathfrak{q}(s)=q|s|/2\pi$, where $s=1/2+it$. 
The drop-off is so quick that one can
obtain a truncation error of at most $\epsilon_1$ simply on stopping the series 
in \eqref{afe Z} after $\approx \sqrt{2\mathfrak{q}
\log(\mathfrak{q}/\epsilon_1)}$ terms. 
Hence, we can choose a uniform truncation point $N:=N_{Q,\Delta}$ for the series
 in \eqref{afe Z} for all quadratic conductors 
$q\in [Q,Q+\Delta)$. This yields the formula 
\begin{equation}\label{afe error}
Z(t,\chi) = 2\re \left[ e^{i\theta(t,\mathfrak{a})}
F(t,\chi)+ \mathcal{R}_1(t,N,q)\right], 
\end{equation}
where
\begin{equation}\label{Ftchi}
F(t,\chi) = \sum_{n=1}^N \frac{\chi(n)}{n^{1/2+it}} 
V_{1/2+it+\mathfrak{a}}(\pi n^2/q).
\end{equation}

We now enforce the bound $\Delta < Q$, as in Theorem~\ref{main
thm}, as well as the bound $N > \sqrt{2Q/\pi}$, which our choice of $N$ will 
ultimately satisfy if $Q$ is large enough.
Then, given $0< \epsilon_1 <1$,  we can ensure 
via Lemma~\ref{R1 lemma} in Section~\ref{proofs of lemmas} that 
\begin{equation}
\mathcal{R}_1(t,N,Q,\Delta) := \max_{q \in
[Q,Q+\Delta)} |\mathcal{R}_1(t,N,q)|<\epsilon_1,
\end{equation}
subject to
\begin{equation}\label{N choice}
N\ge  \sqrt{(2Q/\pi) \log(Q/\epsilon_1)}.
\end{equation}
Note that this choice of $N$ 
satisfies our earlier working assumption 
$N>\sqrt{2Q/\pi}$, provided that $Q > 10^4$ say.

As for the rotation phase $\theta(t,\mathfrak{a})$ in \eqref{afe Z},
its numerical evaluation of 
can be done efficiently using known numerical algorithms
for the Gamma function (e.g. \cite{lanczos,spouge}), or 
using  the asymptotic expansion 
\begin{equation}
\theta(t,\mathfrak{a}) \sim \frac{t}{2}\log\left(\frac{q t}{2\pi e}\right)+
\frac{\pi(2\mathfrak{a}-1)}{8} + \frac{1}{48t}+\cdots,\qquad t\to
\infty,
\end{equation}
for large $t$, which is derived via the Stirling formula.

In summary, to compute $Z(t,\chi)$ for a given $t$ and 
$\chi$ with conductor $q\in [Q,Q+\Delta)$, 
it is enough to focus on the sum $F(t,\chi)$ in \eqref{Ftchi}.

The treatment of $F(t,\chi)$ will be separated into several cases.
To begin, recall that a real primitive character exists to a prime power modulus
if and only if the modulus is an odd prime $p$ in which case 
the character is the quadratic symbol 
\begin{equation}
\left(\frac{n}{p}\right) =
\left(\frac{(-1)^{(p-1)/2}p}{n}\right),
\end{equation}
or the modulus is $4$ in which case the character is
$\chi_4(n) = (-1)^{(n-1)/2}\mathds{1}_{n\,\textrm{odd}}$, 
or the modulus is 
$8$ in which case there are two characters, $\chi_8(n) = (-1)^{(n^2-1)/8}\mathds{1}_{n\,\textrm{odd}}$
and $\chi_4(n)\chi_8(n)= (-1)^{(n^2+4n-5)/8}\mathds{1}_{n\,\textrm{odd}}$; see e.g.\ \cite{davenport}. 
Hence, the only quadratic conductors $q$ 
are products such moduli, subject to the moduli being relatively prime. 
If $q$ is 
not divisible by $8$, then there is one primitive 
quadratic character with modulus $q$, otherwise there are two primitive
characters. 

The cases of even and odd $q$ will be handled separately by our algorithm.
Furthermore, in order 
to fix the form of the functional equation of the $L$-function, 
we deal with the possibilities $\mathfrak{a}=0$
or $\mathfrak{a}=1$ separately also.
Since our algorithm will apply 
to each situation analogously,
we henceforth fix $\mathfrak{a}=0$ and $q$ is odd. This corresponds to 
positive odd fundamental discriminants, which we denote by
$\mathcal{Q}_{\textrm{odd}}^+$.
In particular, there is now a unique $\chi_q$ associated with each $q\in
\mathcal{Q}_{\textrm{odd}}^+$, and we may unambiguously write $F(t,q)$ instead of $F(t,\chi_q)$ for the main sum. 


\section{The Taylor expansion}
We use the Taylor expansion to remove the dependence on $q$ in the weighting
function $V$. To this end, define
\begin{equation}
G_z(w) := \int_1^{\infty} e^{-wy} y^{z-1}dy.
\end{equation}
By definition, we have
\begin{equation}
\begin{split}
\frac{V_{1/2+it}(\pi n^2/q)}{n^{1/2+it}} =
C(t,q) G_{1/4+it/2}(\pi n^2/q),
\quad C(t,q):=\frac{\left(\pi/q\right)^{1/4+it/2}}{\Gamma(1/4+it/2)}.
\end{split}
\end{equation}
We apply the Taylor expansion to $G_z(w)$ in the $w$ variable to obtain 
\begin{equation}\label{G taylor exp}
G_{1/4+it/2}(\pi n^2/q)= \sum_{r=0}^{\infty}
\frac{G^{(r)}_{1/4+it/2}(\pi n^2/Q)}{r!}  \left(\frac{\pi n^2 (Q-q)}{Qq}\right)^r, 
\end{equation}
where $G^{(r)}_z(w)$ is the $r$-th derivative of $G$ with respect to $w$.
Let 
\begin{equation}
c_r(t,n) := \frac{G^{(r)}_{1/4+it/2}(\pi n^2/Q)}{r!} \left(\frac{\pi n^2}{Q}\right)^r,
\end{equation}
and let
\begin{equation}
F_r(t,q) := \sum_{n=1}^N c_r(t,n)\chi_q(n).
\end{equation}
Truncating the series expansion in \eqref{G taylor exp} after $R$ terms
then yields
\begin{equation}\label{R2 formula}
F(t,q) = C(t,q)
\sum_{r=0}^{R-1}  
\left(\frac{Q-q}{q}\right)^r F_r(t,q) +
\mathcal{R}_2(t,N,q,R).
\end{equation}
Furthermore, given $0<\epsilon_2<1$, we can ensure via Lemma~\ref{R2 lemma} in
Section~\ref{proofs of lemmas} that
\begin{equation}
\mathcal{R}_2(t,N,Q,\Delta,R) := \max_{q \in
[Q,Q+\Delta)} |\mathcal{R}_2(t,N,q,R)| < \epsilon_2,
\end{equation}
subject to
\begin{equation}\label{R choice}
R \ge \frac{\log(N/\epsilon_2)}{\log(Q/\Delta)}. 
\end{equation}
Note that $R$ depends only logarithmically on $Q$ and $\epsilon_2$, and so grows
slowly with these parameters.

The importance of the Taylor expansion \eqref{R2 formula} 
is it allows one to avoid having to re-calculate the
weighting function $V$ for each new $q$.
This has a large impact on practical running time, 
as was observed in \cite{stopple2009}.
Indeed, using the recursions described in \cite{stopple2009}, 
the computation of the derivatives of $G_z(w)$ can be reduced 
quickly
to that of computing $G_z(w)$ itself. In turn, the function 
$G_z(w)$ can be computed efficiently 
using one of the many algorithms
for the incomplete Gamma function; see \cite{rubinstein} for a survey.

\begin{remark}
We still remain within the target complexity of Theorem~\ref{main thm}
even if each evaluation of $G_{1/4+it/2}(\pi n^2/Q)$ takes as much as $\Delta/N$ time and
space. This is 
because there are only $N$ values of $G_{1/4+it/2}(\pi n^2/Q)$ that we need to
compute. 
\end{remark}

In summary, 
we may assume that the coefficients $c_r(t,n)$ have been precomputed.

\section{Exponential sums}\label{exp sums sec}

Consider the quadratic exponential sum (Gauss sum) 
\begin{equation}\label{gqn}
g_q(n):= \sum_{\ell=0}^{2n-1} e^{\pi i q\ell^2/n}.
\end{equation}
By results in e.g.\ \cite{davenport,BEW, FJK} we have the Fourier expansion 
\begin{equation}
\chi_q(n) =  
g(q) g_q(2n), \qquad g(q):=\frac{e^{\pi i q(q-2)/8 - \pi i /8}}{2\sqrt{2}}, 
\end{equation}
subject to $q$ being odd and $(n,q)=1$.
Moving the factor $g(q)$, which is independent of $n$, to the outside,  
we see that 
\begin{equation}\label{F sum 2}
F_r(t,q)=g(q)
\sum_{\substack{n=1 \\ (n,q)=1}}^N \frac{c_r(t,n)g_q(2n)}{\sqrt{n}}.
\end{equation}

We wish to remove the summation condition $(n,q)=1$ in \eqref{F sum 2}
because it will prevent 
us from applying the multi-evaluation method in the next section. 
With this in mind, 
note that if $q\in \mathcal{Q}_{\textrm{odd}}^+$, then
$q=p_1\cdots p_{\omega}$ where $p_h$ are distinct odd primes. So letting 
\begin{equation}
\tilde{S}_r(t,q,a) := \sum_{1\le m \le N/a}  \frac{c_r(t,am)g_q(2am)}{\sqrt{am}},
\end{equation}
$\mathcal{P}(q)=\{\textrm{prime } P\,:\, P\,|\,q \textrm{ and } P\le
N\}$, 
and applying inclusion-exclusion, $F_r(t,q)$ can be expressed as
\begin{equation}\label{inc-exc}
F_r(t,q)=g(q)\sum_{h=0}^{|\mathcal{P}(q)|} 
(-1)^h \sum_{\{P_1,\ldots,P_h\}\subseteq \mathcal{P}(q)}\tilde{S}_r(t,q,P_1\cdots P_h).
\end{equation}
If $P_1\cdots P_h > N$, then $\tilde{S}_r(t,q,P_1\cdots P_h) = 0$ 
because the sum is empty.
In particular, we may restrict the inner sum in \eqref{inc-exc}
to subsets $\{P_1,\ldots,P_h\}$ such that the product is $\le N$.
Moreover, using the notation $a=P_1\cdots P_h$ and $b=q/a$,
we have $g_q(2am) = a g_b(2m)$. So, if we let
\begin{equation}\label{tildeS S0}
S_r(t,a,b):= \sum_{1\le m \le N/a}  \frac{c_r(t,am)g_b(2m)}{\sqrt{m}},
\end{equation}
then
\begin{equation}\label{tildeS S}
\tilde{S}_r(t,q,P_1\ldots P_h) = \sqrt{a}S_r(t,a,b). 
\end{equation}

Put together, and in light of the preceding observations,
it suffices to compute the sum $S_r(t,a,b)$ 
for all tipples $(r,a,b)$ such that 
$Q/a \le b < Q/a+\Delta/a$, $1\le a\le N$, and $0\le r< R$.
After that, $Z(t,\chi_q)$ can be recovered for all odd 
quadratic
conductors $q\in
[Q,Q+\Delta)$ by back-substitution into formulas \eqref{tildeS
S}, \eqref{inc-exc}, \eqref{R2 formula}, and \eqref{afe error}. 

\section{Multiple evaluations algorithm}

Consider an exponential sum of the form
\begin{equation}\label{exp sum}
Z(h):=\sum_{k=1}^K a_k e^{2\pi i \alpha_k h},\qquad 0\le \alpha_k < 1,\quad
|a_k|\le 1.
\end{equation}
The Odlyzko--Sch\"onhage algorithm~\cite{odlyzko-schonhage} 
provides an efficient method to compute 
$Z(h)$ for all integers $h\in [-H/2,H/2)$.
Briefly, the problem is transformed via the fast Fourier
transform to that of evaluating a sum of rational functions of the form
\begin{equation}\label{rational func}
\sum_{k=1}^K \frac{\beta_k}{z-e^{-2\pi i \alpha_k}},\qquad 
\beta_k :=a_k e^{-2\pi i \alpha_k}(e^{2\pi i \alpha_k H}-1), 
\end{equation}
at all the $H$-th roots of unity. After that, either the Odlyzko-Sch\"onhage
rational function method, 
which is based on Taylor expansions, or the Greengard--Rokhlin
algorithm~\cite{greengard-rokhlin}, 
which additionally uses certain recursions and translations, 
can be used to perform the rational function evaluation.
The Greengard--Rokhlin algorithm seems to perform better 
in practice; see \cite{gourdon} for an example implementation in the context of
the zeta function.\footnote{There is vast literature on improvements
and generalizations of the Greengard--Rokhlin algorithm; 
see e.g.\ \cite{beatson-greengard} for a survey.}

Either of these methods yields 
an essentially optimal average time (up to logarithmic factors)
for computing $Z(h)$. In fact, 
using the notation so far
and letting $M=\max\{H,K\}$, 
one easily deduces the following from \cite[\textsection{3} \& Theorem 4.1]{odlyzko-schonhage}.

\begin{theorem}[Odlyzko--Sch\"onhage]\label{os algorithm}
The function $Z(h)$ can be computed for all integers $h\in [-H/2,H/2)$ 
with error $< \epsilon_3$ using $\ll M \log^2(M/\epsilon_3)$
arithmetic operations on numbers of $\ll \log(M/\epsilon_3)$ bits
and $\ll M \log^2(M/\epsilon_3)$ bits of storage.
\end{theorem}

One can also use a nonuniform fast Fourier transform method instead, 
which avoids the
intermediate step involving the rational functions \eqref{rational func}; see
\cite{dr93,gl04}. This is 
likely be more efficient in practice.

\section{Proof of Theorem~\ref{main thm}}\label{main thm proof}

\begin{proof}
We assume that $Q>10^4$, otherwise the computation of $F(t,q)$ can be done directly
for each quadratic conductor $q\in [Q,Q+\Delta)$.
In what follows, we will use the estimate (valid for $w > 0$ and
$\re(z)\le 1$)
\begin{equation}
|G_z^{(r)}(w)| \le \int_1^{\infty} e^{-wy} y^r\,dy
\le w^{-r-1} \int_0^{\infty} e^{-u} u^r\,du \le \frac{r!}{w^{r+1}}.
\end{equation}
This estimate implies, in particular, that $|c_r(t,n)| \le  Q/\pi$. 
This bound is implicitly used when we apply Theorem~\ref{os algorithm} next.
(Theorem~\ref{os algorithm}
assumes that the coefficients of the exponential sum  $Z(h)$ are bounded by $\ll 1$,
and more generally the conclusion of the theorem still holds if the coefficients
are bounded by $\ll M$, say.)

We use Theorem~\ref{os algorithm} to bound the complexity 
of computing, to within $\epsilon_3\in (0,1)$, the sum
\begin{equation}
S_r(t,a,b)=\sum_{\substack{1\le m \le N/a \\0\le \ell < 4m}} 
\frac{c_r(t,am) \, e^{\pi i b \ell^2/2m}}{\sqrt{m}}
\end{equation}
for all $b\in [Q/a, Q/a+\Delta/a)$, $a\in [1,N]$, and $r\in [0,R)$. 
Given $a$, the length of the sum $S_r(t,a,b)$ is 
about $2N(N+a)/a^2$ terms,
and the number of points $b$ 
where we want to evaluate it is $\le \Delta/a+1$ points. Therefore, 
by Theorem~\ref{os algorithm}, 
there exists an absolute constant
$A_5$ such that the total cost of computing $S_r(t,a,b)$ for a given $a$ and $r$,
and for all $b\in [Q/a, Q/a+\Delta/a)$, is 
\begin{equation}
\le A_5 (\Delta/a + N^2/a^2)\log^2 ((\Delta+N)/\epsilon_3)
\end{equation}
arithmetic operations on numbers of $\log(\Delta+N)/\epsilon_3)$ bits 
and requiring a similar amount of storage space.
Summing over $a$ and $r$, 
this cost is 
\begin{equation}\label{total cost 2}
\le A_6 R (\Delta \log N+ N^2) \log^2 ((\Delta+N)/\epsilon_3)
\end{equation}
arithmetic operations, 
where $A_6$ is an absolute constant (e.g.\ we can 
take $A_6 =2 A_5$).
We choose $R$ to be the ceiling of the lower bound in \eqref{R choice}.
In addition, we note that $\Delta < 2\pi N^2/\log N$
since $Q>10^4$ by assumption. So there is a constant
$A_7$ such that \eqref{total cost 2} is bounded by 
\begin{equation}\label{total cost 3}
\le A_7 N^2 \frac{\log(N/\epsilon_2)}{\log(Q/\Delta)}
\log^2(N/\epsilon_3)
\end{equation}
Choosing $N$ to be the ceiling of the lower bound in \eqref{N choice}, 
and substituting into \eqref{total cost 3} now gives that this is
\begin{equation}
\le A_8 Q \log(Q/\epsilon_1) \frac{\log(Q/(\epsilon_1\epsilon_2))}{\log(Q/\Delta)}
\log^2(Q/(\epsilon_1\epsilon_3)),
\end{equation}
where $A_8$ is an absolute constant.

We choose $\epsilon_1=\epsilon_2=\epsilon/8$, 
and $\epsilon_3= \epsilon/(2R\sqrt{N Q})$. 
This choice of $\epsilon_3$ 
ensures that $\tilde{S}_r(t,q,a)$ can be
recovered from $S_r(t,a,b)$ 
via formula \eqref{tildeS S}
to within
$\epsilon/(2R\sqrt{Q})$.
In turn, $F(t,q)$ 
can be recovered from $\tilde{S}_r(t,q,a)$ 
via formulas \eqref{inc-exc} and \eqref{R2 formula}
to within $\epsilon/4$. (Here, we used that $|g(q)C(t,q)|\le 0.34$ 
for $t\in [0,1]$,
and we bounded the length of the sum in \eqref{inc-exc} by $d(q)\le \sqrt{Q}$, which is
very generous.)  
Overall, 
the recovery process thus described takes at most an additional $\le A_9 d(q) R$ operations for some absolute constant $A_9$.
Finally, the value of $Z(t,\chi_q)$ can be computed 
 using formula \eqref{afe error}, 
and the accumulated 
(roundoff and truncation) error in computing $Z(t,\chi_q)$ this way 
is $< 2\epsilon_1+2\epsilon_2+R\sqrt{N Q} \epsilon_3\le \epsilon$.

Note that that to execute formula \eqref{inc-exc}, 
we needed information about the factorization of $q$.
Acquiring this information causes little difficulty as 
one can determine the factorization of each $q\in [Q,Q+\Delta)$ 
in poly-log time on average using an obvious generalization of the 
sieve of Eratosthenes. 
\end{proof}

\section{Improvements}\label{improvements}

The most significant improvement to Theorem~\ref{main thm} would be to reduce the
window size $\Delta$ that is required for optimal performance, currently 
$\Delta \asymp Q$. 
One might even hope for 
optimal performance over windows as small as $\Delta \asymp \sqrt{Q}$. 
Unfortunately, it is not clear how to do this, though
a key step in that direction seems to involve repeated applications of
the reciprocity law for the quadratic Gauss sum $g_b(m)$ 
in order reduce the length of the main sum.

One can use various identities for $g_b(m)$ to obtain 
immediate savings. For example, 
\begin{equation}\label{gb id}
g_b(2m) = 4\sum_{\ell = 0}^{m-1} e^{\pi i b \ell^2/2m} + 
2\left\{\begin{array}{ll}
e^{\pi i m/2}-1 & \textrm{if $b\equiv 1\pmod*{4}$,}\\
e^{3\pi i m/2}-1 &  \textrm{if $b\equiv 3\pmod*{4}$.}
\end{array}\right.
\end{equation}
Using this identity enables cutting the length of the sum $S_r(t,a,b)$ from 
about $2(N/a)^2$ terms to about $\frac{1}{2}(N/a)^2$ terms, 
a speed-up by a factor of $4$. 
There may be other identities and symmetries of $g_b(m)$ that 
can be similarly exploited.

Another potentially interesting  improvement is 
to obtain a ``hybrid algorithm'' that works uniformly
in the $t$ and $q$ aspects. One difficulty here is that  
the dependence on $t$ in the main sum 
occurs implicitly, through the weighting function $V$.  
However, as suggested in \cite{rubinstein}, one can 
 express $V$ using a Riemann sum
 and then interchange the order of summation and integration. 
This could produce an exponential sum more amenable to fast evaluation via
the algorithm underlying Theorem~\ref{os algorithm}.

\section{Lemmas}\label{proofs of lemmas}

\begin{lemma}\label{R1 lemma}
Assume that $Q>10^4$ and $Q>\Delta$.
If $N\ge  \sqrt{(2Q/\pi)\log(Q/\epsilon_1)}$, 
then $|\mathcal{R}_1(t,N,Q,\Delta)|< \epsilon_1$.
\end{lemma}
\begin{proof}
First, we recall the definition \eqref{V def} of the weighting
function $V_z(w)$, the integral representation 
\eqref{incomp gamma} of $\Gamma(z,w)$, and the lower bound
$\min_{t\in [0,1]} |\Gamma(1/4+it/2)| > 1$. 
Using this, we deduce that
the size of $n$-th term in the tail of the series \eqref{afe Z} 
is  
\begin{equation}\label{afe Z term}
\le \frac{1}{\sqrt{n}} \left(\frac{\pi n^2}{q}\right)^{\mathfrak{a}/2+1/4} \int_1^{\infty} e^{-\pi
n^2 y/q} y^{\mathfrak{a}/2-3/4}dy < 
\frac{1}{\sqrt{n}}\left(\frac{q}{\pi n^2}\right)^{3/4-\mathfrak{a}/2}
e^{-\pi n^2/q},
\end{equation}
where we arrived at the last inequality 
on using that $\mathfrak{a}\le 1$ together with 
the trivial bound
$y^{\mathfrak{a}/2-3/4} \le 1$ for $y\ge 1$. 
Thus, summing the terms in  the series \eqref{afe Z} over 
$n>N$, bounding each term using \eqref{afe Z term},
and estimating the sum by an integral (c.f.\ \cite[page 320]{weinberger}) gives 
\begin{equation}
|\mathcal{R}_1(t,N,Q,\Delta)|  
\le \frac{1}{2}\left(\frac{q}{\pi}\right)^{7/4} \frac{e^{-\pi N^2/q}}{N^2}.
\end{equation}
Taking $N$ as in the lemma and using the bound $q<2Q$ 
yields the desired bound.
\end{proof}

\begin{lemma}\label{R2 lemma}
Assume $Q>10^4$ and $Q>\Delta$.
If $R \ge  \log(N/\epsilon_2)/\log(Q/\Delta)$,
then $|\mathcal{R}_2(t,N,Q,\Delta,R)|< \epsilon_2$. 
\end{lemma}
\begin{proof}
If we truncate the Taylor series \eqref{G taylor exp} at $r=R$, 
then the remainder can be expressed using the well-known integral 
form for the Taylor remainder. 
By \cite[Lemma 2.1]{stopple2009},  this remainder, call it 
$\upsilon(q,r)$, is of size 
\begin{equation}
|\upsilon(q,r)|\le \frac{(q/Q-1)^R}{R!} \Gamma(R,\pi n^2/q) \le \frac{(q/Q-1)^R}{R}.
\end{equation}
Therefore, recalling \eqref{R2 formula} and using that $|C(t,q)|\le
(\pi/q)^{1/4}$ for $t\in [0,1]$, and $\sum_{1\le n\le N} 1/\sqrt{n} \le
2\sqrt{N}$, we obtain
\begin{equation}
|\mathcal{R}_2(t,N,q,R)| < 2\sqrt{N} \left(\frac{\pi}{q}\right)^{1/4} \frac{(q/Q-1)^R}{R}.
\end{equation}
Since $Q\le q < Q+\Delta$, this gives
\begin{equation}
|\mathcal{R}_2(t,N,Q,\Delta,R)| < \frac{2\sqrt{N}}{R} \left(\frac{\pi}{Q}\right)^{1/4}
\left(\frac{\Delta}{Q}\right)^R.
\end{equation}
So, taking $R$ as in the lemma more than suffices to 
give the desired bound.

\end{proof}

\bibliographystyle{amsplain}
\bibliography{centAlg}
\end{document}